\documentclass[10pt,reqno]{amsart}

\usepackage{amsmath}
\usepackage{amsfonts}
\usepackage{amssymb}
\newcommand{\vertiii}[1]{{\left\vert\kern-0.25ex\left\vert\kern-0.25ex\left\vert #1 
    \right\vert\kern-0.25ex\right\vert\kern-0.25ex\right\vert}}
\usepackage{amsthm}
\usepackage{mathrsfs}
\usepackage{latexsym}

\usepackage{cite} 

\usepackage{esint}

\numberwithin{equation}{section}

\usepackage{graphicx,subfigure}

\usepackage{ifthen} 

\provideboolean{shownotes} 
\setboolean{shownotes}{true} 
\newcommand{\margnote}[1]{
\ifthenelse{\boolean{shownotes}}%
{\marginpar{\raggedright\tiny\texttt{#1}}}%
{}%
}

\newcommand{\hole}[1]{
\ifthenelse{\boolean{shownotes}}%
{\begin{center} \fbox{ \rule {.25cm}{0cm}
\rule[-.1cm]{0cm}{.4cm} \parbox{.85\textwidth}{\begin{center}
\texttt{#1}\end{center}} \rule {.25cm}{0cm}}\end{center}}
{}
}

\usepackage[colorlinks=true,urlcolor=blue,
citecolor=red,linkcolor=blue,linktocpage,pdfpagelabels,
bookmarksnumbered,bookmarksopen]{hyperref}

\theoremstyle{plain}

\newtheorem{lemma}{Lemma}[section]
\newtheorem{theorem}[lemma]{Theorem}

\newtheorem{corollary}[lemma]{Corollary}

\theoremstyle{definition}
\newtheorem{remark}[lemma]{Remark}
\newtheorem{definition}[lemma]{Definition}

\theoremstyle{remark}

\newcommand{\R}{\mathbb{R}}
\newcommand{\C}{\mathbb{C}}

\newcommand{\N}{\mathbb{N}}

\newcommand{\tip}{\widetilde{p}}
\newcommand{\timu}{\widetilde{\mu}}
\newcommand{\tiK}{\widetilde{K}}
\newcommand{\tika}{\widetilde{\kappa}}
\newcommand{\tinu}{\widetilde{\nu}}
\newcommand{\tiA}{\widetilde{A}}
\newcommand{\tiB}{\widetilde{B}}

\newcommand{\cL}{{\mathcal{L}}}
\newcommand{\cD}{{\mathcal{D}}}

\newcommand{\cU}{{\mathcal{U}}}
\newcommand{\cE}{{\mathcal{E}}}

\newcommand{\vep}{\varepsilon}

\renewcommand{\Re}{\mathrm{Re}\,} 
\renewcommand{\Im}{\mathrm{Im}\,}

\newcommand{\bu}{\overline{u}}
\newcommand{\bv}{\overline{v}}
\newcommand{\bU}{\overline{U}}
\newcommand{\bmu}{\overline{\mu}}
\newcommand{\bkp}{\overline{\kappa}}
\newcommand{\bL}{\overline{L}}
\newcommand{\bq}{\overline{q}}
\newcommand{\bB}{\overline{B}}

\newcommand{\bA}{\overline{A}}

\newcommand{\hU}{\widehat{U}}
\newcommand{\hV}{\widehat{V}}

\newcommand{\<}{\langle}
\renewcommand{\>}{\rangle}

\begin{document}

\title[Dissipative structure of isothermal Korteweg compressible fluids]{Dissipative structure of one-dimensional isothermal compressible fluids of Korteweg type}

\author[R. G. Plaza]{Ram\'on G. Plaza}

\address{{\rm (R. G. Plaza)} Instituto de 
Investigaciones en Matem\'aticas Aplicadas y en Sistemas\\Universidad Nacional Aut\'onoma de 
M\'exico\\ Circuito Escolar s/n, Ciudad Universitaria, C.P. 04510\\Cd. de M\'{e}xico (Mexico)}

\email{plaza@mym.iimas.unam.mx}

\author[J. M. Valdovinos]{Jos\'{e} M. Valdovinos}

\address{{\rm (J. M. Valdovinos)} Instituto de 
Investigaciones en Matem\'aticas Aplicadas y en Sistemas\\Universidad Nacional Aut\'onoma de 
M\'exico\\ Circuito Escolar s/n, Ciudad Universitaria, C.P. 04510\\Cd. de M\'{e}xico (Mexico)}


\email{valdovinos94@comunidad.unam.mx}

\begin{abstract}
This paper studies the dissipative structure of the system of equations that describes the motion of a compressible, isothermal, viscous-capillar fluid of Korteweg type in one space dimension. It is shown that the system satisfies the genuine coupling condition of Humpherys \cite{Hu05}, which is, in turn, an extension to higher order systems of the classical condition by Kawashima and Shizuta \cite{KaSh88a,ShKa85} for second order systems. It is proved that genuine coupling implies the decay of solutions to the linearized system around a constant equilibrium state. For that purpose, the symmetrizability of the Fourier symbol is used in order to construct an appropriate compensating matrix. These linear decay estimates imply the global decay of perturbations to constant equilibrium states as solutions to the full nonlinear system, via a standard continuation argument.
\end{abstract}

\keywords{Isothermal Korteweg compressible fluids, genuine coupling, dissipative structure, global decay.}

\subjclass[2010]{35Q35 (primary), 76N99, 35L65, 35B40 (secondary)}

\maketitle

\setcounter{tocdepth}{1}



\section{Introduction}

In this paper we consider the following system of equations,
\begin{equation}
\label{SysLagi}
\begin{aligned}       
	v_t - u_x &= 0, \\
        u_t + p(v)_x &= \Big( \frac{\mu(v)}{v} u_x \Big)_x- \Big( \kappa(v) v_{xx}+\tfrac{1}{2}\kappa'(v)v_x^{2} \Big)_x,\\       
\end{aligned}
\end{equation}
with $x\in \mathbb{R}$, $t>0$, for unknowns $u=u(x,t)$ (velocity) and $v=v(x,t)$ (specific volume), describing the dynamics of a compressible isothermal fluid exhibiting viscosity and capillarity in a one dimensional unbounded domain and in Lagrangian coordinates. The function $p = p(v)$ represents the pressure and  $\mu = \mu(v)$ and $\kappa = \kappa(v)$ correspond to nonlinear viscosity and capillarity coefficients, respectively.

In order to describe fluid capillarity effects in diffuse interfaces for liquid vapor flows, Korteweg \cite{Kortw1901} formulated (based on van der Waals' approach \cite{vdW1894}) a constitutive equation of stress tensors that included density gradients and that was, in general, incompatible with equilibrium thermodynamics. To circumvent this difficulty, the model was later rigorously derived by Dunn and Serrin \cite{DS85} under the viewpoint of rational mechanics by introducing the concept of interstitial work. System \eqref{SysLagi} is the simplified isothermal version of the model derived by Dunn and Serrin in one space dimension and in the case of nonlinear viscosity and capillarity coefficients. Notice that when $\kappa \equiv 0$ system \eqref{SysLagi} reduces to the compressible Navier-Stokes equations. Hence, it is also known as the one-dimensional \emph{Navier-Stokes-Korteweg system} in the literature (cf. \cite{ChZha14,CLS19,BrGL-V19}).

Since Dunn and Serrin's work, the Navier-Stokes-Korteweg system has been the subject of intense research thanks to its rich mathematical structure and diverse applications. Many mathematical results abound in the literature, mainly pertaining to local and global existence of classical solutions \cite{HaLi94,HaLi96a,HaLi96b}, weak solutions \cite{BrDjL03,Hasp09,Hasp11,DD01} and strong solutions \cite{Kot08,Kot10}; to the decay of perturbations of equilibria \cite{WaTa11,TZh14,GLZh20}, or to the description of phase-field transition fronts \cite{Sl83,Sl84,FrKo17,FrKo19,HaSl83}, just to mention a few. The reader is also referred to the recent works \cite{KoTs21,KSX21,KMS22}. In sum, the mathematical literature on Korteweg models is very vast and we only mention an abridged list of references. In this contribution, we examine the \emph{dissipative structure} of the isothermal Navier-Stokes-Korteweg system in one dimension.

The seminal works by Kawashima \cite{KaTh83} and by Kawashima and Shizuta \cite{KaSh88a,KaSh88b,ShKa85} established the theoretical basis to study the strict dissipativity of a large class of second order systems of the generic form,
\begin{equation}
\label{ssecord}
U_t + F(U)_x = (B(U)U_x)_x + Q(U),
\end{equation}
which encompass a large number of models such as the Navier-Stokes and Navier-Stokes-Fourier systems, relaxation models, and many other model equations of hyperbolic-parabolic type in the continuum mechanics literature (see, e.g., the survey articles \cite{Ka86,XuKa17}, as well as \cite{KY04,UDK12,UKS84,KY09,HaNa03,Y04} and the numerous references therein). Kawashima and Shizuta showed that strict dissipativity, or the property that solutions to the linearized equations around any equilibrium state exhibit some time decay structure (for the precise statement, see \cite{ShKa85}), is tantamount to a simple algebraic \emph{genuine coupling condition} expressing that no eigenvector of the hyperbolic part of the linearized operator lies on the kernel of the dissipative terms (viscous and/or relaxation). In their theory, the symmetrizability of the systems (in the classical sense of Friedrichs \cite{Frd54}) plays a key role. Genuine coupling allows the existence of \emph{compensating matrices}, which are very useful to obtain energy estimates both at linear and nonlinear levels. Among its many consequences, it is possible to prove the nonlinear time decay of small perturbations around equilibrium states for generic genuinely coupled second order systems.

In a later work, Humpherys \cite{Hu05} extended Kawashima and Shizuta's notions of strict dissipativity, genuine coupling and symmetrizability to viscous-dispersive one-dimensional systems of \emph{any} order. First, Humpherys introduces the concept of \emph{symbol symmetrizability}, which is a generalization of the classical notion of term-wise symmetrizability, and extends the genuine coupling condition for any symmetric Fourier symbol of the linearized higher-order operator around an equilibrium state. Humpherys then shows that his notion of genuine coupling is equivalent to the strict dissipativity of the system and to the existence of a compensating matrix \emph{symbol} for the linearized system. The potential applications are numerous and include viscous-dispersive systems such as the Korteweg model under consideration. Surprisingly, there is no study available in the literature that analyzes whether genuine coupling for higher order systems can be applied to obtain the decay of perturbations to equilibrium states, not even at the linear level. 

The purpose of this paper is, therefore, very concrete: to examine the dissipative structure of the Korteweg system \eqref{SysLagi} in the sense of Humpherys and to apply such structure (more precisely, the compensating matrix symbol) in order to obtain the time-decay of solutions, both at the linear and the nonlinear levels. Hence, our contributions are mainly of methodological nature and can be summarized as follows:
\begin{itemize}
\item[-] It is shown that the Navier-Stokes-Korteweg system \eqref{SysLagi} is symbol symmetrizable (although not Friedrichs symmetrizable) and that it satisfies the genuine coupling condition.
\item[-] Thanks to a new degree of freedom in the choice of the symbol symmetrizer it is possible to construct an appropriate compensating matrix symbol for the linearized system (see Lemma \ref{lemourK} below), skew-symmetric and bounded above uniformly in the Fourier parameter, which allows to close the energy estimates at the linear level (Lemma \ref{lembee}). 
\item[-] It is shown that the linear decay estimates can be used to prove the nonlinear decay of small perturbations of constant equilibrium states, culminating into the global existence and optimal time-decay of perturbation solutions (see Theorem \ref{gloexth} below). 
\end{itemize}

It is to be noticed that, although it is similar in spirit to its second order counterpart, the proof of the linear decay rates for the system (Lemma \ref{lembee}) does not follow automatically from the hyperbolic-parabolic result. The main reason is that the system is of third order and that both the symmetrizer and the compensating matrix are now symbols and depend on the Fourier parameter as well. Albeit Humpherys' equivalence theorem (see Theorem 6.3 in \cite{Hu05}) implies the existence of a compensating matrix symbol, in order to close the energy estimates we choose an appropriate symbol suitable for our needs. Our choice is simple and particular to the Korteweg system \eqref{SysLagi}. We conjecture that for generic, higher-order systems, the symbol should be constructed for different Fourier (or wave number) parameter regimes.

It is also important to observe that the final existence and global decay result (Theorem \ref{gloexth}) is not new. For example, Wang and Tan \cite{WaTa11} obtained optimal decay rates for the system in dimension three (and we conjecture that the result in one space dimension can be obtained by adapting their method of proof; see also \cite{GLZh20,TZh14} for related results). The analyses in these works, however, consist of obtaining decay rates for the linearized problem and then performing nonlinear energy estimates, both steps relying heavily on the intrinsic structure of the Korteweg model. The technique presented here is based solely on the genuine coupling condition and on the appropriate construction of the compensating matrix symbol. 
%
%
%
%
%
%
%
Hence, the description of how to obtain global decay from the dissipative structure of the underlying linearized Korteweg system compensates, we hope, for the lack of novelty of some of the results. We actually believe that the present analysis has prospects for the study of the decay structure of larger, more complicated systems, or with higher order derivatives, for which a direct approach may seem cumbersome or difficult to tackle. For example, the analysis of the full heat-conducting Korteweg system with nonlinear viscosity and capillarity coefficients will be addressed in a companion paper \cite{PlVa2b}.


The paper is structured as follows. Section \ref{seckorteweg} contains a description of the equations for a one-dimensional Korteweg compressible fluid, as well as the main hypotheses of the paper. The known local existence theory is reviewed in section \ref{localexist}. In section \ref{secgencoup} we examine the genuine coupling condition for the system of equations and exhibit a family of symbol symmetrizers. With this information, in section \ref{seclineardecay} we obtain the linear decay rates for the associated semigroup. Finally, section \ref{secglobal} contains the global decay result for small perturbations of constant equilibrium states.

\subsection*{On notation} 

We denote the real and imaginary parts of a complex number $\lambda \in \C$ by $\Re\lambda$ and $\Im\lambda$, respectively, as well as complex conjugation by ${\lambda}^*$. The standard inner product of two vectors $a$ and $b$ in $\C^n$ will be denoted by $\langle a,b \rangle = \sum_{j=1}^n a_j b_j^*$. Complex transposition of matrices are indicated by the symbol $A^*$, whereas simple transposition is denoted by the symbol $A^\top$. Standard Sobolev spaces of functions on the real line will be denoted as $H^s(\R)$, with $s \in \R$, endowed with the standard inner products and norms. The norm on $H^s(\R)$ will be denoted as $\| \cdot \|_s$ and the norm in $L^2$ will be henceforth denoted by $\| \cdot \|_0$. Any other $L^p$ -norm will be denoted as $\| \cdot \|_{L^p}$ for $p \geq 1$. 

\section{One-dimensional isothermal compressible fluids of Korteweg type}
\label{seckorteweg}

In Eulerian coordinates, the system of equations describing isothermal compressible fluids in one space dimension has the form

\begin{equation}
\label{SysEul}
\begin{aligned}
        \partial_t\rho +\partial_{y}(\rho u) &= 0,   \\
        \partial_t(\rho u) + \partial_{y}\left(\rho u^{2}+ \tip(\rho) \right) &= \partial_{y} \big( \timu(\rho)u_{y} + \tiK \big),\\       
\end{aligned} 
\end{equation}
where $y\in \mathbb{R}$ and $t>0 $ denote position and time, respectively, and the unknown scalar functions $\rho= \rho(y,t)$ and $u= u (y,t)$ denote the denote the density and velocity fields, respectively. Here, the thermodynamic pressure, $\tip = \tip(\rho)$, is a smooth function of the density and the density-dependent viscosity coefficient, $\timu= \timu(\rho):= 2 \nu (\rho)+ \lambda (\rho) > 0$, is a short-cut for the bulk and kinematic viscosity coefficients, $\nu$ and $\lambda$, respectively, satisfying $2\nu + 3\lambda \geq 0$, $\nu >0$ for all $\rho$. Finally, $\tiK$ is the Korteweg tensor, having the form
\begin{equation}
\label{KortwTensor}
\tiK := \rho \tika(\rho) \partial_{y}^{2}\rho+ \tfrac{1}{2}\big( \rho \tika'(\rho)- \tika(\rho) \big) \rho_{y}^{2},
\end{equation}
where $\tika = \tika(\rho)>0$ is the (smooth) capillary coefficient. See, e.g., \cite{CCD15, HaLi94, DS85, Kortw1901, vdW1894} and the references therein. It is further assumed that $\tip, \timu, \tika \in C^{\infty}\left( (0,\infty) \right)$, i.e. that they are smooth functions of $\rho \in (0, \infty)$ and that they satisfy $\tip, \timu, \tika >0$.

For simplicity, let us recast (\ref{SysEul}) and (\ref{KortwTensor}) in Lagrangian coordinates. For that purpose, we define the Lagrangian variable as
\[
x(y, t) := \int_0^{y} \rho(\xi, t) \: d \xi,
\]
and the specific volume as $v := 1/\rho >0$. Under this change of variables it is easy to verify that the equation of conservation of mass (first equation in (\ref{SysEul})) transforms into
\begin{equation}
\label{MassLag}
v_t= u_x.
\end{equation}
There is a systematic way to transform from Eulerian to Lagrangian coordinates back and forth. Following Serre \cite{Ser1}, any balance law in Eulerian coordinates of the form
\[
\partial_tW_{j}+ \partial_{y}\Phi_{j}=0,
\]
for extensive quantities $W_{j}$ and $\Phi_{j}$, is equivalent to
\[ 
\partial_t(v W_{j})+ \partial_x(\Phi_{j}- uW_{j})=0,
\]
in Lagrangian coordinates. This property holds because the differential form $(\Phi_{j}-uW_{j})dt- v W_{j} dx$ is exact as the reader may easily verify. For example, taking $W_{4} \equiv 1$ and $\Phi_{4} \equiv 0$ we arrive at the continuity equation \eqref{MassLag}. The momentum equation in Lagrangian coordinates can be obtained by considering $W_{3}= \rho u$ and $\Phi_{3}= \rho u^{2}+ \tip(\rho) - \timu(\rho)u_{y}-\tiK$. After some straightforward algebra the resulting equation is
\begin{equation}
\label{MomLag}
u_t+p(v)_x = \Big( \frac{\mu(v)}{v} u_x \Big)_x- \Big( \kappa(v) v_{xx} + \tfrac{1}{2} \kappa'(v) v_x^{2} \Big)_x,
\end{equation} 
where we have defined
\begin{equation}
\label{CofLag}
\begin{aligned}       
	p(v) &:= \tip(1/v), \\
       \mu(v) &:= \timu(1/v),  \\ 
        \kappa(v) &:=  \frac{\tika(1/v)}{v^5},
\end{aligned}
\end{equation}
as $C^{\infty}$ functions of the specific volume $v \in (0, \infty)$. Details are left to the reader (see \cite{BDD06,CCD15,Ser1} for further information). Henceforth, the system of equations describing one-dimensional isothermal compressible flow with viscosity and capillary in Lagrangian coordinates reads,
\begin{equation}
\label{SysLag}
\begin{aligned}       
	v_t - u_x &= 0, \\
        u_t + p(v)_x &= \Big( \frac{\mu(v)}{v} u_x \Big)_x- \Big( \kappa(v) v_{xx}+\tfrac{1}{2}\kappa'(v)v_x^{2} \Big)_x,\\       
\end{aligned}
\end{equation}
with $x\in \mathbb{R}$, $t>0$, for the unknown velocity, $u=u(x,t)$, and specific volume, $v=v(x,t)$.

In view that the process is isothermal, the thermodynamic variables depend upon the specific volume alone. We make the following assumptions:
\begin{itemize}
\item[(H$_1$)] \phantomsection\label{H1} The specific volume is the only independent thermodynamic variable. It belongs to the open domain 
\[
\cD = \left\lbrace v \in \R \: : \: 0 < C_0^{-1} < v < C_0 \right\rbrace,
\]
for some uniform positive constant $C_0$.\\
\item[(H$_2$)] \phantomsection\label{H2} The pressure, the combined viscosity and the capillary coefficients are smooth functions of $v \in \cD$, $p, \mu, \kappa \in C^{\infty} (\overline{\cD})$, satisfying $p, \mu, \kappa > 0$ for all $v \in \cD$.\\
\item[(H$_3$)] \phantomsection\label{H3} There exists a finite number of open subsets $\cU_j \subset \mathcal{D}_0$, $1 \leq j \leq N$, $N \in \N$, such that $\bigcup_{j=1}^N \cU_j \subset \cD$ and satisfying $p'(v) < 0$ for all $v \in \cU_j$, $j=1, \ldots, N$.
\end{itemize}

\begin{remark}
Hypothesis \hyperref[H3]{(H$_3$)} guarantees the hyperbolicity in each open set $\cU_j$ of the underlying ``inviscid" system
\[
\begin{aligned}       
	v_t - u_x &= 0, \\
        u_t + p(v)_x &= 0,
\end{aligned}
\]
also known as the $p$-system. This assumption allows us to consider models for phase transitions such as the isothermal van der Walls equation of state,
\[
p(v) = \frac{R \theta}{v -b} - \frac{a}{v^{2}},
\]
for some positive constants $a, b, R, \theta >0$. For temperatures below a critical value, $\theta < \theta_{c}= 8a/27bR$, the shape of $p$ exhibits a spinodal region $v \in (\alpha, \beta)$, where $p'(v) > 0$, and which separates the liquid and vapor phases. Thus, for the van der Walls model we have $N = 2$ and $\cU_1=(b,\alpha)$ (liquid phase), $\cU_{2}=(\beta, C_0)$ (vapor phase), for some $0< b < \alpha < \beta$ and with $\cD = \left\lbrace v \: : \: b < v < C_0 \right\rbrace$. Another typical example is the adiabatic gas law
\[
p(v) = \frac{R\theta}{v^\gamma}, 
\]
for some constant $\gamma > 1$, so that $N = 1$, $\cU_{1} = \cD= \left\lbrace v \: : \: C_0^{-1} < v < C_0 \right\rbrace$ because $p'(v) = -\gamma R \theta v^{-(\gamma+1)}< 0$ for all $v \in \cD$.
\end{remark}

\begin{remark}
Hypothesis \hyperref[H2]{(H$_2$)} encompasses the most general form of the viscosity and capillary coefficients. Typical choices found in the literature for the Eulerian coefficients as functions of the density are of the algebraic type, $\timu(\rho)= \mu_0\rho^{\alpha}$ and $\tika(\rho)= \kappa_0\rho^{\beta}$, for some $\alpha, \beta >0$ and constants $\mu_0$, $\kappa_0$. These choices yield
\begin{equation}
\label{physcoeffalg}
\kappa(v)= \frac{\kappa_0}{v^{\beta +5}}, \quad \mu(v)= \frac{\mu_0}{v^{\alpha}}
\end{equation}
(see, e.g., \cite{CCD15}). It is also common to find $\kappa(v)\equiv \kappa_0 >0$, that is, $\tika(\rho)=\kappa_0 \rho^{-5}$, and $\mu(v)\equiv \mu_0 >0$ with $\kappa_0, \mu_0$ constants (see \cite{Hu09}). In \cite{BrGL-V19}, the authors propose a relation between the two coefficients, by defining the Eulerian capillary coefficient $\tika(\rho)$ as
\[
\tika(\rho) = \frac{\tinu'(\rho)^{2}}{\rho}, 
\]
where $\tinu(\rho)=\rho^{(s+3)/2}$, $s \in \R$ (bulk viscosity). In \cite{ChHa13a} the authors consider, for example, viscosity coefficients of St. Venant type satisfying $\timu(\rho)\geq \rho^\alpha$ with $0 \leq \alpha < 1/2$. In this paper (like, for instance, in \cite{LiLuo16}), we make the general assumption \hyperref[H2]{(H$_2$)}.
\end{remark}

\section{Local well-posedness of perturbations of equilibrium states}
\label{localexist}
The local well-posedness of system \eqref{SysLag} (or equivalently, of system \eqref{SysEul}) is well understood. See, for example, the analysis of Hattori and Li \cite{HaLi94} in the case of viscosity and capillary coefficients of algebraic type. The global existence was studied by the same authors in \cite{HaLi96a, HaLi96b}. The work of Benzoni-Gavage \emph{et al.} \cite{BDD06} focuses on the purely capillar case (with $\mu=0$), whereas Danchin and Dejardins \cite{DD01} consider constant viscosity and capillary coefficients in critical Besov spaces. Variable (or density-dependent) coefficients have been studied in \cite{HaLi94, ChHa13a,CCD15}. In sum, the literature is vast and the local existence of classical solutions is now a well-established fact. Here we follow the presentation of Chen \emph{et al.} \cite{CCD15}, who recast local well-posedness in terms of perturbations of an equilibrium state, a formulation which is suitable for our needs.

Let $(\bv,\bu)\in \cU_r \times \mathbb{R}  $, for some fixed $1\leq r \leq N$, be a constant equilibrium state. Let us consider solutions to system (\ref{SysLag}) of the form $(\bv+v,\bu+u)$ where, with a slight abuse of notation, $v$ and $u$ now represent perturbations of the constant state $(\bv, \bu)$. Upon substitution into (\ref{SysLag}), one obtains the following nonlinear system for the perturbation $(v, u)$:
\begin{equation}
\label{PertSys}
\begin{aligned}
        v_t-u_x &= 0,   \\
        u_t + p(\bv+v)_x &= \Big( \frac{\mu(\bv+v)}{\bv+v} u_x \Big)_x - \Big(\kappa(\bv+ v) v_{xx} + \tfrac{1}{2}\kappa'(\bv+v)v_x^{2} \Big)_x,\\       
\end{aligned}
\end{equation}
assuming, of course, that the perturbation $v$ is such that $\bv+v \in \cU_r$ (the perturbed specific volume remains in one open phase $\cU_r$). Under assumption \hyperref[H2]{(H$_2$)} let us denote
\[
0 < \sup_{v \in \cU_r} \mu(v) \leq \max_{\overline{\cD}} \mu(v) =: \mu_0 < \infty, \quad 0 < \sup_{v \in \cU_r} \kappa(v) \leq \max_{\overline{\cD}} \kappa(v) =: \kappa_0 < \infty.
\]

For positive constants $M\geq m >0$, $T>0$ and for any $s \geq 3$ we denote the function space,
\[
\begin{aligned}
X_s((0,T);m,M) := \Big\{ (v,u) \, : \; \, &v-\bv \in C((0,T); H^{s+1}(\R)) \cap C^1((0,T); H^{s-1}(\R)), \\
& u-\bu \in  C((0,T); H^s(\R)) \cap C^1((0,T); H^{s-2}(\R)),\\
&(v_x, u_x) \in L^2((0,T); H^{s+1}(\R) \times H^s(\R)), \\
&\text{and} \; m \leq v(x,t) \leq M \; \text{a.e. in } \, (x,t) \in \R \times (0,T) \Big\}.
\end{aligned}
\]
%
For $U := (v+ \bv, u+ \bu)^\top \in X_{s}\left((0,T); m, M \right)$ and any $0 \leq t_1 \leq t_2 \leq T$ we define
\begin{equation}
\label{tripnorm}
\vertiii{U}^2_{s,[t_{1},t_{2}]} := \sup_{t_{1} \leq t \leq t_{2}} \left(\Vert v(t)\Vert_{s+1}^{2}+ \Vert u(t) \Vert_{s}^{2}  \right) + \int_{t_{1}}^{t_{2}} \left(  \Vert v_x(t)\Vert_{s+1}^{2} + \Vert u_x(t)\Vert_{s}^{2} \right)\: dt, 
\end{equation}
as well as
\[
\vertiii{U}_{s,T}:= \vertiii{U}_{s,[0,T]}.
\]

Therefore, the following local existence theorem for perturbations can be established (see \cite{CCD15, HaLi94} and Remark 1.1 in \cite{HaLi96a}). We omit its proof for the sake of brevity.
\begin{theorem}[local existence]
\label{thmlocale}
Under hypotheses \hyperref[H1]{(H$_1$)} - \hyperref[H3]{(H$_3$)}, suppose that $(\bv,\bu)\in  \cU_r \times \R $, for some fixed $1\leq r \leq N$, is a constant equilibrium state. Suppose that $v_0\in H^{s+1}(\R)$ and $u_0\in H^{s}(\R)$, $s \geq 3$, satisfy
\begin{equation}\label{thlocex-hyp0}
0 < m_0 \leq \bv + v_0(x) \leq M_0, \qquad \bv + v_{0}(x) \in \cU_r,
\end{equation}
a.e. in $x \in \R$ for some positive constants $M_0 \geq m_0 > 0$. Then there exists $T > 0$ depending only on $\kappa_0, \mu_0, m_0, M_0, \| u_0 \|_s$ and $\| v_0 \|_{s+1}$ such that the Cauchy problem for system \eqref{PertSys} with initial data $v(0) = v_0+ \bv$, $u(0)=u_0+ \bu$ admits a unique smooth solution $U = (v+ \bv, u + \bu)^\top \in X_{s}\left((0,T);\tfrac{1}{2} m_0, 2M_0 \right)$ satisfying
\begin{equation}
\label{est.loc.exis}
\vertiii{U}_{s,T} \leq C \left(  \Vert v_0 \Vert_{s+1}^{2} + \Vert u_0 \Vert_{s}^{2} \right)^{1/2},
\end{equation}
for some constant $C > 1$ depending only on $m_0$ and $M_0$.

%
%
%
\end{theorem}

\begin{remark}
Notice that, by Sobolev imbedding theorem, the local solution satisfies $v + \bv \in \cU_r$ as long as the initial perturbation $(v_{0}, u_{0})$ is small enough. Also $\mu_{0}$, $\kappa_{0}$, $m_{0}$ and $M_{0}$ depend on such initial perturbation. Hence, for convenience, we recast Theorem \ref{thmlocale} in terms of perturbation variables.
\end{remark}

\begin{theorem}[local existence of small perturbations]
\label{thmlocalepert}
Under hypotheses \hyperref[H1]{(H$_1$)} - \hyperref[H3]{(H$_3$)}, assume that $(\bv,\bu)\in  \cU_r \times \R $, for some fixed $1\leq r \leq N$, is a constant equilibrium state. Suppose the initial perturbation of the constant equilibrium state satisfies $\left(v_{0}, u_{0} \right) \in H^{s+1}(\R) \times H^{s}(\R)$, for $s \geq 3$. Then we can find a positive constant $a_{0}$ such that if 
\begin{equation}
\label{smallpert}
\left( \Vert v_{0} \Vert_{s+1}^{2} + \Vert u_{0} \Vert_{s}^{2} \right)^{1/2} \leq a_{0},
\end{equation}
then there is a time $T_{1}= T_{1}(a_{0})$ such that a unique solution $(v + \bv, u + \bu) ^\top \in X_{s}\left((0,T_{1});\tfrac{1}{2} m_0, 2M_0 \right)$ (here $m_{0}$ and $M_{0}$ satisfying \eqref{thlocex-hyp0}) exists for the Cauchy problem \eqref{PertSys} with initial condition $v(0)=v_{0}+\bv$, $u(0)=u_{0}+\bu$, and the estimate
\begin{equation}
\label{est.loc.exis 1}
\vertiii{U}_{s,T_{1}} \leq C_{0} \left(  \Vert v_0 \Vert_{s+1}^{2} + \Vert u_0 \Vert_{s}^{2} \right)^{1/2},
\end{equation}
holds for some constant $C_{0}>1$ depending only on $a_{0}$. 
\end{theorem}

\begin{corollary}[a priori estimate]
\label{thaprest}
Under the hypothyeses of Theorem \ref{thmlocale}, let $(v+ \bv, u+ \bu)$ be the solution to Cauchy problem for system \eqref{PertSys} with initial data $v(0) = v_0+\bv$, $u(0)=u_0+ \bu$. Then there exists a constant $0 < a_{2}\ll 1 $, such that if 
$$  \sup_{0 \leq t \leq T} ( \Vert v(t) \Vert_{s+1}^{2} + \Vert u(t) \Vert_{s}^{2} )^{1/2} \leq a_{2},  $$
then
\begin{equation}\label{aprest}
\vertiii{U}_{s,T} \leq C_{2} \left(  \Vert v_{0} \Vert_{s+1}^{2} + \Vert u_{0} \Vert_{s}^2 \right)^{1/2}.  
\end{equation} 
Here $C_{2}= C_{2}(a_{2})$ is independent of $T$.
\end{corollary}
\begin{proof}
Follows from the local existence result, Theorem \ref{thmlocalepert} (see, e.g., the proof of Theorem 2.2 in \cite{HaLi96a}).
\end{proof}
%
%
%
%
%
\begin{remark}
Theorem \ref{thmlocale} is essentially the local well-posedness result in \cite{CCD15} (see Proposition 2.1) which is, in turn, based on the local analysis of Hattori and Li \cite{HaLi94}. The latter is an application of the contraction mapping principle and \emph{a priori} estimates. Even though the results in \cite{CCD15} and \cite{HaLi94} are expressed in terms of physical coefficients of algebraic type of the form \eqref{physcoeffalg}, the analysis can be carried out under the more general hypothesis \hyperref[H2]{(H$_2$)} at the expense of extra bookkeeping (for details, see \cite{ValTh20}). 
\end{remark}
%
%
%
%
%
%
%

\section{The genuine coupling condition}
\label{secgencoup}

In this section we examine the genuine coupling condition for system \eqref{PertSys} within the framework of the extension to viscous-dispersive systems due to Humpherys \cite{Hu05} of the strict dissipativity condition by Kawashima and Shizuta \cite{KaSh88a,ShKa85}. We show that this dissipative structure implies, in turn, the linear decay of solutions to the linearized problem.

\subsection{Linearization and symbol symmetrizability}
\label{linzedpob}
Let $\bU := (\bv, \bu)^\top \in \cU_r \times \R$ be a constant equilibrium state for some fixed $1 \leq r \leq N$ and consider the nonlinear system \eqref{PertSys} for perturbations $U = (v,u)^\top$. Expanding the nonlinear terms
%
we recast system \eqref{PertSys} as
\begin{equation}
\label{nlop}
U_t = \cL U + \partial_x H,
\end{equation}
where $\cL$ is the linearized operator around $\bU$ given by
\begin{equation}
\label{linop}
\cL := \sum_{j=1}^3 \bL_j \partial_x^j,
\end{equation}
with constant coefficient matrices, $\bL_j \in \R^{2 \times 2}$, namely
\[
\bL_1 := \begin{pmatrix} 0 & 1 \\ \bq & 0
\end{pmatrix}, \qquad 
\bL_2 := \begin{pmatrix} 0 & 0 \\ 0 & \bmu / \bv
\end{pmatrix}, \qquad 
\bL_3 := - \, \begin{pmatrix} 0 & 0 \\ \bkp & 0
\end{pmatrix},
\]
and where $\bkp := \kappa(\bv)$, $\bmu := \mu(\bv)$ and $\bq := - p'(\bv)$ are positive constants. The nonlinear terms have the form,
\[
H(U,U_x,U_{xx}) = \begin{pmatrix} 0 \\ H_2(U,U_x,U_{xx})
\end{pmatrix},
\]
with
\begin{align}
H_2(U, U_x, U_{xx}) &:= -(p(\bv + v) - p(\bv) - p'(\bv)v) +  \left( \frac{\mu(\bv+v)}{\bv+v} u_x - \frac{\mu(\bv)}{\bv} u_{x} \right) + \nonumber\\ &\qquad + (\kappa(\bv+v)v_{xx} - \kappa(\bv) v_{xx}) + \tfrac{1}{2} \kappa'(\bv + v) v_x^2 \nonumber\\
&= O(v^2 + |v||u_x| + |v||v_{xx}| + v_x^2).\label{orderH2}
\end{align}
It is to be observed that the nonlinear terms have divergence form.

Let us study the linearized system around $\bU$, namely,
\begin{equation}
\label{ccoeffs}
U_t = \cL U = \sum_{j=1}^3 \bL_j \partial_x^j U.
\end{equation}
Take the Fourier transform of the constant coefficients equation \eqref{ccoeffs} to obtain
\begin{equation}
\label{fouriersys}
\hU_t = \big( i \xi \bL_1 - \xi^2 \bL_2 - i \xi^3 \bL_3 \big) \hU, \qquad \xi \in \R.
\end{equation}

Following Humpherys \cite{Hu05}, let us split the symbol into odd and even terms, and define
\begin{equation}
\label{exprAB}
\begin{aligned}
A(\xi) &:= - \sum_{j \text{ odd}} (i\xi)^{j-1} \bL_j = - \bL_1 + \xi^2 \bL_3 = \begin{pmatrix} 0 & -1 \\ - (\bq + \xi^2 \bkp) & 0\end{pmatrix},\\
B(\xi) &:= - \sum_{j \text{ even}} (-1)^{j/2} \xi^j \bL_j = \xi^2 \bL_2 = \xi^2 \begin{pmatrix} 0 & 0 \\ 0 & \bmu / \bv \end{pmatrix} =: \xi^2 \bB,
\end{aligned}
\end{equation}
with $\bB \geq 0$ constant matrix, positive semi-definite. Therefore, the Fourier equation is recast as
\begin{equation}
\label{fouriersys2}
\hU_t + i \xi A(\xi) \hU + B(\xi) \hU = 0.
\end{equation}
Hence, the evolution of solutions to the linear system \eqref{ccoeffs} is determined by the following family of eigenvalue problems parametrized by $\xi$ and associated to \eqref{fouriersys2}, 
\begin{equation}
\label{evaluesys}
\lambda \hU = - \big(i \xi A(\xi) + B(\xi)\big) \hU.
\end{equation}

A constant coefficient system of the form \eqref{ccoeffs} is symmetrizable in the classical sense of Friedrichs \cite{Frd54} if, for any fixed but arbitrary state $\bU$, there exists a symmetric, positive definite matrix $\bA_{0} = \bA_0(\bU)$ such that $\bA_{0}\bL_j$ is symmetric for all $j$. Humpherys \cite{Hu05} observed that the (linearized) Korteweg system \eqref{ccoeffs} is not Friedrichs symmetrizable, but it can be symmetrized in the following sense.
\begin{definition}[symbol symmetrizability \cite{Hu05}]
\label{defsymH}
The operator $\cL$ is symmetrizable if there exists a smooth, symmetric matrix-valued function $A_{0}=A_{0}(\xi)>0$, positive-definitive, such that both $A_{0}(\xi)A(\xi)$ and $A_{0}(\xi)B(\xi)$ are symmetric and $A_{0}(\xi)B(\xi)$ is positive semi-definite, $A_{0}(\xi)B(\xi) \geq 0$, for all $\xi \in \R$.
\end{definition}

\begin{lemma}
\label{lemHumsymm}
The linear Korteweg operator $\cL$ in \eqref{linop} is symbol symmetrizable but not Friedrichs symmetrizable. Moreover, every symbol symmetrizer has the form
\begin{equation}
\label{gensymm}
A_0(\xi) = \begin{pmatrix} a(\xi) & 0 \\ 0 & a(\xi) \beta(\xi)^{-1} \end{pmatrix} \in C^\infty(\R;\R^{2 \times 2}), 
\end{equation}
where
\begin{equation}
\label{defbeta}
\beta(\xi) := \bq + \xi^2 \bkp > 0, \qquad \xi \in \R,
\end{equation}
and $a = a(\xi) \in C^\infty(\R)$ is any smooth positive function uniformly bounded below, $a(\xi) \geq a_0 > 0$ for all $\xi \in \R$.
\end{lemma}
\begin{proof}
Follows from direct calculations that yield
\[
A_0(\xi) A(\xi) = \begin{pmatrix} a(\xi) & 0 \\ 0 & a(\xi) \beta(\xi)^{-1} \end{pmatrix} \begin{pmatrix} 0 & -1 \\ -\beta(\xi) & 0 \end{pmatrix} = -a(\xi) \begin{pmatrix} 0 & 1 \\ 1 & 0 \end{pmatrix},
\]
clearly symmetric, and
\[
A_0(\xi) B(\xi) = \begin{pmatrix} a(\xi) & 0 \\ 0 & a(\xi) \beta(\xi)^{-1} \end{pmatrix} \begin{pmatrix} 0 & 0 \\ 0 & \xi^2 \bmu/\bv \end{pmatrix} = \xi^2 \begin{pmatrix} 0 & 0 \\ 0 & a(\xi) \beta(\xi)^{-1} \bmu/\bv \end{pmatrix} \geq 0,
\]
symmetric and positive semi-definite. By definition, $A_0$ is smooth, symmetric and positive definite. That $\cL$ is not symmetrizable in the sense of Friedrichs follows by inspection. Indeed, if
\[
\bA_0 = \begin{pmatrix} a_1 & a_2 \\ a_2 & a_3 \end{pmatrix},
\]
is a (Friedrichs) symmetrizer then the simultaneous symmetrizability of $\bL_j$, $j = 1,2,3$, and $\bq, \bkp, \bmu, \bv > 0$ yield $\bA_0 \equiv 0$, a contradiction, as the reader may easily verify. Likewise, any other symbol symmetrizer $A_0(\xi)$ must be of the form \eqref{gensymm}, again, by direct inspection.
\end{proof}

\begin{remark}
Clearly, every symmetrizable system in the sense of Friedrichs is symbol symmetrizable, but the converse is not true. The linear Korteweg operator \eqref{linop} is a simple and physically relevant counterexample that exhibits the importance of Definition \ref{defsymH}.
\end{remark}

In view of Lemma \ref{lemHumsymm} we multiply system \eqref{fouriersys2} on the left by any symbol symmetrizer in the family to obtain the following symmetric system in the Fourier space,
\[
A_0(\xi) \hU_t + i \xi A_0(\xi) A(\xi) \hU  + \xi^2 A_0(\xi) \bB \hU = 0.
\]
In this fashion, even though the linear system \eqref{ccoeffs} is not symmetrizable, its symbol can be written in symmetric form.

\subsection{Genuine coupling and the equivalence theorem}
Once the system is symmetrized in the sense of Humpherys, one may ask whether there exists a compensating matrix for it. 
\begin{definition}[compensating matrix]
\label{ComMat}
Let $A_{0}$, $A$, $B \in C^{\infty} \left( \R; \R^{2 \times 2} \right)$ be smooth real matrix-valued functions of $\xi \in \R$. Assume that $A_{0},$ $A$, $B$ are symmetric for all $\xi \in \R$, $A_{0} >0$ is positive definite and $B \geq 0$ (positive semi-definite). A real matrix valued function  $K\in C^{\infty} \left( \R; \R^{2 \times 2} \right)$ is said to be a \emph{compensating matrix function} for the triplet $(A_{0}, A, B)$ if 
\begin{itemize}
\item[(a)] $K(\xi)A_{0}(\xi)$ is skew-symmetric for all $\xi\in \R$; and,
\item[(b)] $\left[K(\xi)A(\xi)\right]^{s}+ B(\xi) \geq \theta(\xi) I > 0$ for all $\xi \in \R$, $\xi \neq 0$, some $\theta = \theta(\xi) > 0$.
\end{itemize}
Here $[M]^{s} := \frac{1}{2}(M+M^\top)$ denotes the symmetric part of any real matrix $M$. 
\end{definition}

\begin{definition}
\label{defstrictdiss}
Consider the Fourier system \eqref{fouriersys2} and its associated eigenvalue problem \eqref{evaluesys}.
\begin{itemize}
\item[(i)]$\mathcal{L}$ is called \emph{strictly dissipative} if for each $\xi \neq 0$ we have that all solutions to \eqref{evaluesys} satisfy $\Re  \lambda (\xi) < 0$.
\item[(ii)]$\mathcal{L}$ is said to satisfy the \emph{genuine coupling condition} if no eigenvalue of $A(\xi)$ lies in $\ker B(\xi)$ for all $\xi \neq 0$. 
\end{itemize}
\end{definition}

Following the seminal ideas of Kawashima and Shizuta \cite{KaSh88a,ShKa85} in the context of hyperbolic-parabolic systems, Humpherys \cite{Hu05} extended to higher-order, viscous-dispersive systems (like the Korteweg model under consideration) the notions of symmetrizability, genuine coupling and strict dissipativity, culminating into the following
\begin{theorem}[equivalence theorem \cite{Hu05}]
\label{HuThSym}
Suppose that there exists is a symbol symmetrizer, $A_0 = A_{0}(\xi)$, for the operator $\mathcal{L}$ in the sense of Definition \ref{defsymH}, and that $A_{0}(\xi)A(\xi)$ is of constant multiplicity in $\xi$. Then the following conditions are equivalent:
\begin{itemize}
\item[(i)] $\mathcal{L}$ is strictly dissipative.
\item[(ii)] $\mathcal{L}$ is genuinely coupled.
\item[(iii)] There exists a compensating matrix function for the triplet $(A_0,A_0A, A_0B)$.
\end{itemize}
\end{theorem}

\begin{remark}
The equivalence theorem (see Theorem 3.3 or Theorem 6.3 in \cite{Hu05}) holds for any linearized system of arbitrary order. It essentially links the strict dissipativity of the linearized operator (a property expressed in terms of the essential spectrum of the symbol of the operator) with the existence of a compensating matrix symbol which may be useful to prove energy estimates. The difference with previous results (which pertain to second order systems that are symmetrizable in the sense of Friedrichs) is that now symmetrizability involves a symbol and that the compensating matrix depends on the Fourier wave numbers as well. 
\end{remark}

\begin{lemma}
\label{lemgencoup}
The linearized Korteweg operator $\cL$ in \eqref{linop} satisfies the genuine coupling condition. Moreover, for any symbol symmetrizer of the form \eqref{gensymm}, the matrix symbol $A_0(\xi) A(\xi)$ is of constant multiplicity in $\xi$.
\end{lemma}
\begin{proof}
Follows directly from the expressions \eqref{exprAB}, which yield $\ker A_{0}(\xi)B(\xi) = \ker A_{0}(\xi) \bB =\ker \bB = \text{span} \{ (1,0)^\top\}$, if $\xi \neq 0$, and
\[
A_{0}(\xi)A(\xi) \begin{pmatrix} 1 \\ 0 \end{pmatrix} = -a(\xi) \begin{pmatrix} 0 \\ 1 \end{pmatrix} \neq \zeta(\xi) \begin{pmatrix} 1 \\ 0 \end{pmatrix},
\]
for any eigenvalue $\zeta(\xi)$ of $A(\xi)$. Hence, no eigenvector of $A_{0}(\xi)A(\xi)$ lies in $\ker A_{0}(\xi) B(\xi)$ for $\xi \neq 0$.

To prove the second assertion, it is to be observed that, by elementary computations, the eigenvalues of $A_0(\xi) A(\xi)$ are given by $\tilde{\zeta}(\xi)_{1,2} = \pm a(\xi)^{1/2}$. Since $a(\xi) \geq a_0 > 0$ for all $\xi \in \R$, both eigenvalues are real and simple for each value of $\xi$. The lemma is proved.
\end{proof}

In view of Lemma \ref{lemgencoup}, the Korteweg system under consideration satisfies the hypotheses of the equivalence theorem and we may readily conclude the existence of a compensating matrix symbol for it. Since the system is of third order, however, in order to make it useful to perform energy estimates we need to choose it with a bit of extra care.

\section{Linear decay}
\label{seclineardecay}

In this section we establish the decay estimates for solutions to the linear Korteweg system \eqref{ccoeffs}. The main ingredient is the selection of the compensating matrix symbol. Thanks to the degree of freedom in choosing a symbol symmetrizer for the linear operator, we are able to find (by direct inspection) an appropriate compensating matrix function suitable for our needs. 

\subsection{The compensating matrix symbol}
\label{seccompfunc}

From Lemma \ref{lemHumsymm} let us choose the symbol symmetrizer with $a(\xi) \equiv \beta(\xi) \geq \bq >0 $, so that from this point on we define
\begin{equation}
\label{defA0}
A_0(\xi) := \begin{pmatrix} \beta(\xi) & 0 \\ 0 &  1 \end{pmatrix}.
\end{equation}
Let us introduce the following change of variables,
\begin{equation}
\label{varV}
\hV := A_0(\xi)^{1/2} \hU.
\end{equation}
Hence, equation \eqref{fouriersys2} is transformed into
\begin{equation}
\label{eqV}
\hV_t + i\xi \tiA(\xi) \hV + \tiB(\xi) \hV = 0,
\end{equation}
where
\[
\begin{aligned}
\tiA(\xi) := A_0(\xi)^{1/2} A(\xi) A_0(\xi)^{-1/2} &= \begin{pmatrix} \beta(\xi)^{1/2} & 0 \\ 0 &  1 \end{pmatrix} \begin{pmatrix} 0 & -1 \\ -\beta(\xi) & 0 \end{pmatrix} \begin{pmatrix} \beta(\xi)^{-1/2} & 0 \\ 0 &  1 \end{pmatrix}\\
&= \begin{pmatrix} 0 & -\beta(\xi)^{1/2} \\ -\beta(\xi)^{1/2} & 0 \end{pmatrix},
\end{aligned}
\]
is smooth and symmetric, and where
\[
\begin{aligned}
\tiB(\xi) := A_0(\xi)^{1/2} B(\xi) A_0(\xi)^{-1/2} &= \begin{pmatrix} \beta(\xi)^{1/2} & 0 \\ 0 & 1 \end{pmatrix} \begin{pmatrix} 0 & 0 \\ 0 & \xi^2 \bmu/\bv \end{pmatrix} \begin{pmatrix} \beta(\xi)^{-1/2} & 0 \\ 0 &  1 \end{pmatrix}\\
&= \xi^2 \begin{pmatrix} 0 & 0 \\ 0 & \bmu/\bv \end{pmatrix} \\&= \xi^2 \bB,
\end{aligned}
\]
is clearly symmetric and positive semi-definite.

The following lemma chooses appropriately the compensating matrix symbol for the system in the new coordinates.
\begin{lemma}
\label{lemourK}
There exists a smooth, skew-symmetric, real matrix valued function $\tiK = \tiK(\xi) \in C^\infty(\R; \R^{2 \times 2})$ such that it is a compensating matrix for the triplet $(I, \tiA(\xi), \bB)$. In other words, 
\begin{equation}
\label{compmatprop}
[\tiK(\xi) \tiA(\xi)]^s + \bB \geq \overline{\theta} I > 0, 
\end{equation}
for some uniform $\overline{\theta} > 0$ independent of $\xi \in \R$. Moreover, there exists a uniform constant $C > 0$ such that
\begin{equation}
\label{tiKbded}
| \tiK(\xi) | \leq C,
\end{equation}
for all $\xi \in \R$.
\end{lemma}
\begin{proof}
Let us define
\begin{equation}
\label{deftiK}
\tiK(\xi) := \frac{\bmu}{4 \beta(\xi)^{1/2} \bv} \begin{pmatrix} 0 & -1 \\ 1 & 0\end{pmatrix}.
\end{equation}
It is then clear that $\tiK$ is smooth, $\tiK \in C^\infty(\R; \R^{2 \times 2})$, and that $\tiK(\xi)^\top = - \tiK(\xi)$ (skew-symmetric). Since,
\[
\beta(\xi)^{-1/2} = \big(\bq + \xi^2 \bkp \big)^{-1/2} \leq \bq^{-1/2},
\]
for all $\xi \in \R$, we conclude that there exists a uniform constant $C > 0$ such that $|\tiK(\xi)| \leq C$ for all $\xi \in \R$. Finally, noticing that
\[
\tiK(\xi) \tiA(\xi) = \frac{\bmu}{4 \beta(\xi)^{1/2} \bv} \begin{pmatrix} 0 & -1 \\ 1 & 0\end{pmatrix} \begin{pmatrix} 0 & -\beta(\xi)^{1/2} \\ -\beta(\xi)^{1/2} & 0 \end{pmatrix} = \frac{\bmu}{4 \bv} \begin{pmatrix} 1 & 0 \\ 0 & -1 \end{pmatrix},
\]
is symmetric and constant, we then have
\begin{align*}
[\tiK(\xi) \tiA(\xi)]^s + \bB = \tiK(\xi) \tiA(\xi) + \bB &= \frac{\bmu}{4 \bv} \begin{pmatrix} 1 & 0 \\ 0 & -1 \end{pmatrix} + \frac{\bmu}{\bv} \begin{pmatrix} 1 & 0 \\ 0 & 1 \end{pmatrix}\\
&=\frac{\bmu}{4 \bv} \begin{pmatrix} 1 & 0 \\ 0 & 3 \end{pmatrix}\\
&\geq \frac{\bmu}{4 \bv} I,
\end{align*}
and \eqref{compmatprop} holds with $\overline{\theta} = \bmu / (4 \bv) > 0$, uniform in $\xi$. The lemma is proved.
\end{proof}

\subsection{Basic energy estimate}
\label{secbasicest}

The following result provides the basic energy estimate for the linearized equations. The proof follows the same ideas of the corresponding result for linearized hyperbolic-parabolic systems which can be found in \cite{KaTh83,UKS84} (see also \cite{AnMP20}). However, given that we are now dealing with symbols instead of matrices evaluated at constant states, we present the full proof for completeness, following the previous analyses as a blueprint and paying attention to the points where the dependence in the Fourier parameter $\xi \in \R$ plays a role. The precise form of the compensating matrix in \eqref{deftiK} is a key ingredient.

\begin{lemma}[basic energy estimate]
\label{lembee}
The solutions $\hV = \hV(\xi,t)$ to system \eqref{eqV} satisfy the estimate
\begin{equation}
\label{bestV}
|\hV(\xi,t)| \leq C |\hV(\xi,0)| \exp \left( - \frac{k \xi^2 t}{1+ \xi^2}\right),
\end{equation}
for all $\xi \in \R$, $t \geq 0$ and some uniform constants $C,k > 0$.
\end{lemma}
\begin{proof}
Take the real part of the $\C^2$-inner product of $\hV$ with equation \eqref{eqV}. The result is
\begin{equation}
\label{la8}
\tfrac{1}{2} \partial_t |\hV|^2 + \xi^2 \< \hV, \bB \hV \> = 0,
\end{equation}
in view that $\tiB(\xi) = \xi^2 \bB$. Now multiply equation \eqref{eqV} by $- i \xi \tiK(\xi)$ and take the inner product with $\hV$. This yields,
\begin{equation}
\label{la9}
- \< \hV, i \xi \tiK \hV_t \> + \xi^2 \< \hV, \tiK \tiA \hV \> - \< \hV, i \xi^3 \tiK \bB \hV \> = 0.
\end{equation}
Since $\tiK$ is skew-symmetric we clearly have 
\[
\Re \< \hV, i \xi \tiK \hV_t \> = \tfrac{1}{2} \xi \partial_t \< \hV, i \tiK \hV \>.
\]
Taking the real part of \eqref{la9} we then obtain
\[
- \tfrac{1}{2} \xi \partial_t \< \hV, i \tiK \hV \> + \xi^2 \< \hV, [\tiK \tiA]^s \hV \> = \Re \big( i \xi^3 \< \hV, \tiK \bB \hV \> \big).
\]

By symmetry and using that $\bB \geq 0$, $[\tiK \tiA]^s = \tiK \tiA$ is a constant symmetric matrix, and that $\tiK(\xi)$ is uniformly bounded for all $\xi$, we arrive at the estimate
\begin{equation}
\label{la10}
- \tfrac{1}{2} \xi \partial_t \< \hV, i \tiK \hV \> + \xi^2 \< \hV, \tiK \tiA \hV \> \leq \vep \xi^2 |\hV|^2 + C_\vep \xi^4 \< \hV, \bB \hV \>
\end{equation}
for any $\vep > 0$ and where $C_\vep > 0$ is a uniform constant depending only on $\vep > 0$, $|\bB^{1/2}|$ and the constant $C > 0$ in \eqref{tiKbded}. 

Now multiply equation \eqref{la8} by $1 + \xi^2$, equation \eqref{la10} by $\delta > 0$ and add them up to obtain,
\[
\begin{aligned}
\tfrac{1}{2} \partial_t \Big[ (1 + \xi^2) |\hV|^2 - \delta \xi \< \hV, i \tiK \hV \> \Big] + \xi^4 \< \hV, \bB \hV \> &+ \xi^2 \Big[ \delta \< \hV, \tiK \tiA \hV \> + \> \hV, \bB \hV \> \Big] \\ &\leq \vep \delta \xi^2 |\hV|^2 + \delta C_\vep \xi^4 \< \hV, \bB \hV \>.
\end{aligned}
\]
Let us define the \emph{energy},
\[
\cE := |\hV|^2 - \frac{\delta \xi}{1+ \xi^2} \< \hV, i \tiK \hV \>.
\]
It is easy to verify that the quantity $\cE$ is real in view that $\tiK$ is skew-symmetric. Moreover, there exists $\delta_0 > 0$ sufficiently small such that one can find a uniform constant $C_1 > 0$ for which
\[
C_1^{-1} |\hV|^2 \leq \cE \leq C_1 |\hV|^2,
\]
provided that $0 < \delta < \delta_0$. Hence, $\cE$ is indeed an energy, equivalent to $|\hV|^2$, for $\delta > 0$ sufficiently small.

From Lemma \ref{lemourK} we know there exists a uniform $\overline{\theta} > 0$, independent of $\xi \in \R$, for which \eqref{compmatprop} holds for all $\xi \in \R$. Therefore, if $0 < \delta < 1$ we get
\[
\< \hV, (\delta \tiK \tiA + \bB) \hV \> \geq \delta \overline{\theta} |\hV|^2.
\] 
Choosing $\vep = \overline{\theta}/2 > 0$ and $0 < \delta < \min \{ 1, \delta_0, 1/C_\vep\}$, all uniform constants in $\xi \in \R$, we arrive at
\[
\tfrac{1}{2} \partial_t \cE + \tfrac{1}{2} \Big( \frac{\xi^2}{1+\xi^2}\Big) \delta \overline{\theta} |\hV|^2 + \frac{(1- \delta C_\vep)}{1+ \xi^2} \xi^4 \< \hV, \bB \hV \> \leq 0.
\]
Therefore we obtain
\[
\partial_t \cE + \frac{k \xi^2}{1+ \xi^2} \cE \leq 0,
\]
where $k = \delta \overline{\theta}/ C_1 > 0$. This yields the result.
\end{proof}

As a consequence of the basic energy estimate \eqref{bestV} and the transformation \eqref{varV}, we readily obtain the following appropriate decay for the original variables, yielding the right regularity for the specific volume and the velocity field.
\begin{corollary}
The solutions $\hU(\xi,t) = \big(\hU_1(\xi,t), \hU_2(\xi,t)\big)^\top$ to system \eqref{fouriersys2} satisfy the estimate
\begin{equation}
\label{bestU}
(1+\xi^2) |\hU_1(\xi,t)|^2 + |\hU_2(\xi,t)|^2 \leq C \big[(1+\xi^2) |\hU_1(\xi,0)|^2 + |\hU_2(\xi,0)|^2 \big] \exp \left( - \frac{2k \xi^2 t}{1+ \xi^2}\right),
\end{equation}
for all $\xi \in \R$, $t \geq 0$ and some uniform constants $C > 0$ and $k > 0$.
\end{corollary}
\begin{proof}
Let $\hU(\xi,t)$ be a solution to \eqref{fouriersys2}. In view of the transformation \eqref{varV} there holds
\[
|\hV |^{2} = | A_{0}(\xi)^{1/2} \hU|^{2} = 
\left| \begin{pmatrix} \beta(\xi)^{1/2} & 0 \\ 0 & 1 \end{pmatrix} \hU \right|^2 = \beta(\xi) |\hU_1|^2 + |\hU_2|^2. 
\]
Now, since \eqref{fouriersys2} gets transformed into \eqref{eqV} and, by Lemma \ref{lembee}, $\hV$ satisfies estimate \eqref{bestV}, we obtain
\[
\beta(\xi) |\hU_1(\xi,t) |^2 + |\hU_2(\xi,t) |^2 \leq C \big( \beta(\xi) |\hU_1(\xi,0) |^2 + |\hU_2(\xi,0) |^2 \big) \exp \left( - \frac{2k \xi^2 t}{1+ \xi^2}\right),
\]
for some $C > 0$. Observing that there exist uniform constants $C_2 = \min \{ \bq, \bkp\} > 0$ and $C_3 = \max \{ \bq, \bkp \} >0$ such that $C_2 (1 + \xi^2) \leq \beta(\xi) \leq C_3 (1 + \xi^2)$ for all $\xi$, we obtain the result.

\end{proof}

\subsection{Linear semigroup decay rates}
\label{linsemdec}
As a consequence of the basic (pointwise) energy estimate \eqref{bestU} in the Fourier space we obtain the desired decay estimate for the solutions to the linearized system \eqref{ccoeffs}. This is the content of the following
\begin{lemma}
\label{lemlindecay}
Suppose that $U = (U_1, U_2)^\top$ is a solution to the linear system \eqref{ccoeffs} with initial data $U(x,0) \in \big( H^{s+1}(\R) \times H^s(\R) \big) \cap \big( L^1(\R) \times L^1(\R) \big)$ for some $s \geq 2$. Then for each fixed $0 \leq \ell \leq s$ there holds the estimate,
\begin{equation}
\label{linest}
\begin{aligned}
\Big( \| \partial_x^{\ell} U_1(t) \|_1^2 + \| \partial_x^{\ell} U_2(t) \|_0^2 \Big)^{1/2} &\leq C e^{-c_1t} \Big( \| \partial_x^{\ell} U_1(0) \|_1^2 
+ \| \partial_x^{\ell} U_2(0) \|_0^2 \Big)^{1/2}  +\\ & \;\; + C (1+t)^{-(\ell/2 + 1/4)} \| U(0) \|_{L^1},
\end{aligned}
\end{equation}
for all $t \geq 0$ and some uniform constants $C, c_1 > 0$. 
\end{lemma}
\begin{proof}
Fix $\ell \in [0,s]$, multiply \eqref{bestU} by $\xi^{2 \ell}$ and integrate in $\xi \in \R$. The result is
\[
\int_\R \Big[ \xi^{2 \ell}(1+\xi^2) |\hU_1(\xi,t)|^2 + \xi^{2 \ell} |\hU_2(\xi,t)|^2 \Big] \, d\xi \leq C J_1(t) + C J_2(t),
\]
where,
\[
\begin{aligned}
J_1(t) &:= \int_{-1}^1 \Big[ \xi^{2 \ell}(1+\xi^2) |\hU_1(\xi,0)|^2 + \xi^{2 \ell} |\hU_2(\xi,0)|^2 \Big] \exp \left( - \frac{2k \xi^2 t}{1+ \xi^2}\right) \, d \xi,\\
J_2(t) &:= \int_{|\xi|\geq 1} \Big[ \xi^{2 \ell}(1+\xi^2) |\hU_1(\xi,0)|^2 + \xi^{2 \ell} |\hU_2(\xi,0)|^2 \Big] \exp \left( - \frac{2k \xi^2 t}{1+ \xi^2}\right) \, d \xi.
\end{aligned}
\]

To estimate $J_1(t)$, notice that if $|\xi| \leq 1$ then $\exp (-2kt\xi^2/(1+\xi^2)) \leq \exp(-kt\xi^2)$ and we get
\[
J_1(t) \leq 2 \int_{-1}^1 \xi^{2 \ell} |\hU(\xi,0)|^2 e^{-kt \xi^2} \, d\xi \leq 2 \sup_{\xi \in \R} |\hU(\xi,0)|^2 \int_{-1}^1 \xi^{2 \ell} e^{-kt \xi^2} \, d\xi.
\]
But since for any fixed $\ell \in [0,s]$ and constant $k > 0$, the integral
\begin{equation}
\label{I0}
I_0(t) := (1+t)^{\ell + 1/2} \int_{-1}^1 \xi^{2 \ell} e^{-kt \xi^2} \, d\xi \leq C,
\end{equation}
is uniformly bounded for all $t > 0$ with some constant $C > 0$ (see Lemma \ref{lemintbded} below), we obtain
\begin{equation}
\label{estJ1}
J_1(t) \leq C (1+t)^{-(\ell + 1/2)} \| U(x,0)\|_{L^1}^2.
\end{equation}

On the other hand if $|\xi| \geq 1$ then $\exp (-2kt\xi^2/(1+\xi^2)) \leq \exp(-kt)$. Therefore, from Plancherel's theorem we obtain
\[
\begin{aligned}
J_2(t) &\leq e^{-kt} \int_{|\xi|\geq 1} \xi^{2 \ell}(1+\xi^2) |\hU_1(\xi,0)|^2 + \xi^{2 \ell} |\hU_2(\xi,0)|^2 \, d\xi \\
&= e^{-kt} \int_{|\xi|\geq 1} ( \xi^{2\ell}+ \xi^{2(\ell +1)} ) |\hU_1(\xi,0)|^2 + \xi^{2 \ell} |\hU_2(\xi,0)|^2 \, d\xi \\
&\leq  e^{-kt} \int_\R (\xi^{2\ell}+ \xi^{2(\ell+1)}) |\hU_1(\xi,0)|^2 + \xi^{2 \ell} |\hU_2(\xi,0)|^2 \, d\xi \\
&=  e^{-kt} \big( \| \partial_x^{\ell} U_1(0) \|_1^2 + \| \partial_x^{\ell} U_2(0) \|_0^2 \big),
\end{aligned}
\]
for all $t > 0$. Combining both estimates we obtain the result with $c_1 = k/2 > 0$ and the lemma is proved.
\end{proof}


\begin{remark} As a consequence of Lemma \ref{lemlindecay}, we can easily obtain estimate \eqref{linest} for perturbed solutions (that is, replacing $U$ and $U(x,0)$ by $U-\bU$ and $U(x,0)-\bU$, respectively), which is the form needed in Section \ref{secglobal} (details are left to the reader).
\end{remark}

The semigroup associated to the solutions to the linear problem \eqref{ccoeffs} can be expressed in terms of the inverse Fourier transform,
\[
(e^{t \cL} f)(x) = \frac{1}{\sqrt{2 \pi}} \int_\R e^{i x \xi}e^{t M(i\xi)} \hat{f}(\xi) \, d \xi,
\]
where 
\[
M(z) := \begin{pmatrix} 0 & z \\ z (\bq - z^2 \bkp) & z^2 \bmu/\bv\end{pmatrix}, \quad z \in \C,
\]
\[
M(i\xi) = - (i \xi A(\xi) + \xi^2 \bB), \qquad \xi \in \R,
\]
and $U(x,t) = (e^{t \cL} f)(x)$ is the solution to \eqref{ccoeffs} with initial condition $f = (f_1, f_2)^\top$. Therefore, it follows from Lemma \ref{lemlindecay} the linear decay estimate for the semigroup.

\begin{corollary}
\label{corsgdecay}
For any $f \in \big( H^{s+1}(\R) \times H^s(\R) \big) \cap \big( L^1(\R) \times L^1(\R) \big)$, $s \geq 2$, and all $0 \leq \ell \leq s$, $t > 0$, there holds
\begin{equation}
\label{linestsg}
\begin{aligned}
\Big( \| \partial_x^{\ell} (e^{t \cL} f)_1(t) \|_1^2 + \| \partial_x^{\ell} (e^{t \cL} f)_2(t) \|_0^2 \Big)^{1/2} &\leq C e^{-c_1t} \Big( \| \partial_x^{\ell } f_1 \|_1^2 + \| \partial_x^{\ell} f_2 \|_0^2 \Big)^{1/2}  +\\ & \;\; + C (1+t)^{-(\ell/2 + 1/4)} \| f \|_{L^1},
\end{aligned}
\end{equation}
for some uniform $C, c_1 > 0$.
\end{corollary}

\section{Global decay of perturbations of equilibrium states}
\label{secglobal}
Once we have obtained decay rates for the solution to the linearized problem \eqref{linop} around the constant state $\bU$, in this section we prove the global existence of solutions to the nonlinear problem \eqref{nlop} and provide a decay rate of perturbations of the constant equilibrium state as time goes to infinity. 

\subsection{Nonlinear decay rate of solutions} In this section, we obtain the decay rate for the solution to the initial value problem of the nonlinear system \eqref{nlop}, for which we know there exists a local solution by Theorem \ref{thmlocale}. Let us write again the nonlinear system \eqref{nlop}, namely
\begin{equation}
\label{nlop2}
U_t = \cL U + \partial_x H,
\end{equation}
with $\cL $ and $H$ as in Section~\ref{linzedpob}, and consider the initial condition
\begin{equation}\label{InCondNlop}
U(x,0) = U_{0}(x):=\left(v_{0}(x)+ \bv, u_{0}(x)+\bu \right)^\top.
\end{equation}
Under the assumptions of the local existence Theorem \ref{thmlocale} (or Theorem \ref{thmlocalepert}), let us further suppose that 
\[
\left(v_{0}, u_{0} \right)^\top \in \left( H^{s+1}(\R) \times H^{s}(\R)\right) \cap \left( L^{1}(\R) \times L^{1}(\R ) \right),
\]
for some $s \geq 3$. Using the semigroup, $e^{t \cL}$, defined in Section \ref{linsemdec}, we recast the local solution, $U = (U_1, U_2)^\top = (v+ \bv,u+ \bu)^\top$ to \eqref{nlop2} and \eqref{InCondNlop} as  
\begin{equation}
\label{SemSol}
U(x,t)= e^{t \cL}U_{0}(x)+ \int_{0}^{t} e^{(t-z)\cL}(H(x,z)_{x}) \, dz,
\end{equation}
where $H(x,z) = (0, H_2(x,z))^\top$, where $H_2$ is defined in \eqref{orderH2}. Thus, apply the decay of the semigroup (estimate \eqref{linestsg}) to the solution \eqref{SemSol} for some $0 \leq \ell \leq s-1$, in order to obtain the inequality
\begin{equation}
\label{star1}
\begin{aligned}
\big( \Vert \partial_{x}^{\ell}v(t) \Vert_{1}^{2} + \Vert \partial_{x}^{\ell}u(t) \Vert_{0}^{2} \big)^{1/2} 
&\leq C  e^{-c_{1}t}\big( \Vert  \partial_{x}^{\ell}v_{0} \Vert_{1}^{2} +\Vert \partial_{x}^{\ell}u_{0} \Vert^{2}_0 \big)^{1/2} + \\
&\quad + C( 1+t )^{-(1/4 + \ell/2)} \Vert (v_{0}, u_{0}) \Vert_{L^{1}} +\\
&\quad + C\int_{0}^{t} \Vert \partial_{x}^{\ell}(e^{(t-z)\cL} (0, H_{2}(x, z)_{x})^\top) \Vert_0 \: dz . 
\end{aligned}
\end{equation}
From the (easy to verify) identity
\[
\partial_{x}^{\ell}e^{(t-z)\cL} \begin{pmatrix} 0 \\ H_{2}(x,z)_{x} \end{pmatrix} = \partial_{x}^{\ell+1}e^{(t-z)\cL} \begin{pmatrix} 0 \\ H_{2}(x,z) \end{pmatrix},
\]
and upon application of \eqref{linestsg} with $\ell +1 \leq s$ replacing $\ell$, we obtain
\begin{equation}
\label{star2}
\begin{aligned}
\int_{0}^{t} \Vert \partial_{x}^{\ell}(e^{(t-z)\cL} (0, H_{2}(x, z)_{x})^\top) \Vert_0 \: dz &\leq C \int_0^t e^{-c_1(t-z)} \| \partial^{\ell+1} H_2(\cdot,z) \|_0 \, dz + \\
&\quad + C \int_0^t (1+t-z)^{\ell/2 + 3/4} \| H_2(\cdot,z) \|_{L^1} \, dz.
\end{aligned}
\end{equation}
Hence, combining \eqref{star1} and \eqref{star2} yields
\begin{equation}
\label{DerEstNlop}
\begin{aligned}
\big( \Vert \partial_{x}^{\ell}v(t) \Vert_{1}^{2} &+ \Vert \partial_{x}^{\ell}u(t) \Vert_{0}^{2} \big)^{1/2} 
\leq \\
&\leq C\left( e^{-c_{1}t} \big( \Vert  \partial_{x}^{\ell}v_{0} \Vert_{1}^{2} + \Vert \partial_{x}^{\ell}u_{0} \Vert^{2} \big)^{1/2} + ( 1+t )^{-(1/4 + \ell /2)} \Vert (v_{0}, u_{0}) \Vert_{L^{1}} \right) + \\ 
& + C\int_{0}^{t}\left( e^{-c_{1}(t-z)} \Vert \partial_{x}^{\ell+1}H_{2}(\cdot, z) \Vert_{0} + \left(1 + t-z \right)^{-(3/4+\ell/2)}  \Vert H_{2}(\cdot, z) \Vert_{L^{1}} \right) \: dz.
\end{aligned}
\end{equation}
Notice that the divergence form of the nonlinear term is crucial to obtain the algebraic time decay with exponent $-(3/4 + \ell/2)$ inside  the integral.
%

For the sake of brevity, in the sequel we use the following notation. Let us define,
\[
\Vert U(t) \Vert_{k} := \left(\Vert v(t) \Vert_{k+1}^{2} + \Vert u(t) \Vert_{k}^{2}  \right)^{1/2},
\]
for all $0 \leq k \leq s$, and $\Vert U_{0} \Vert_{L^{1}} = \Vert (v_{0}, u_{0}) \Vert_{L^{1}}$.
%
%
Therefore, summing up estimate \eqref{DerEstNlop} for $\ell = 0,1, \ldots, s-1$ we obtain
\begin{equation}\label{nonestHs-1}
\begin{split}
\Vert U(t) \Vert_{s-1} &\leq C(1+t)^{-1/4}\left( \Vert U_{0}\Vert_{s-1} + \Vert U_{0} \Vert_{L^{1}} \right) + \\ 
&\quad\: + \int_{0}^{t}\left( e^{-c_{1}(t-z)}\Vert H_{2}(\cdot, z)\Vert_{s}+ (1+t-z)^{-3/4}\Vert H_{2}(\cdot, z)\Vert_{L^{1}} \right)\: dz.
\end{split}
\end{equation} 

Now, we estimate $\Vert H_{2}(\cdot, z) \Vert_{s}$ and $\Vert H_{2}(\cdot, z) \Vert_{L^{1}}$. First, from \eqref{orderH2} we know that for $U$ close to $\bU$ there holds
\[
H_2(U, U_x, U_{xx}) = O(v^2 + |v||u_x| + |v||v_{xx}| + v_x^2).
\]
Next, we estimate the term $v_{x}^{2}$ with a little care. For this purpose, let us recall that for $u_{1}, u_{2}\in H^{m}(\mathbb{R})\cap L^{\infty}(\mathbb{R})$ we have 
\[
\Vert u_{1}u_{2} \Vert_{m} \leq C \left( \Vert u_{1} \Vert_{m}\| u_{2} \|_{L^{\infty}} + \Vert u_{2} \Vert_{m}\| u_{1} \|_{L^{\infty}} \right)
\]
(see, e.g., Lemma 3.2 in \cite{HaLi96a}). Therefore, by Sobolev imbedding theorem, we get
\[
\Vert v_{x}^{2}(z) \Vert_{s} \leq 2C\Vert v_{x}(z)\Vert_{s} \| v_{x}(z) \|_{L^{\infty}} \leq C \Vert v_{x}(z) \Vert_{s+1} \Vert v(z) \Vert_{2},
\]
for all $z \in [0,t]$. Upon application of Sobolev technical calculus inequalities (see, for example, Lemmata 2.3 and 2.5, pp. 15--16, and equations (3.25), p. 48, in \cite{KaTh83}), we arrive at
\begin{equation}
\label{H2-s}
\begin{split}
\| H_{2}(\cdot, z) \|_{s} &\leq C \big( \Vert v(z) \Vert_{s}^{2} + \Vert v(z) \Vert_{s}\Vert v_{xx}(z) \Vert_{s} + \Vert v(z) \Vert_{s} \Vert u_{x}(z) \Vert_{s} +  \Vert v(z) \Vert_{2} \Vert v_{x}(z) \Vert_{s+1} \big)
\\ &\leq C \left( \Vert v(z) \Vert_{s}^{2}+ \Vert v(z) \Vert_{s} \Vert u_{x}(z) \Vert_{s} + \Vert v(z) \Vert_{s}\Vert v_{x}(z) \Vert_{s+1}  \right),
\end{split}
\end{equation}
and 
\begin{equation}\label{H2-L1}
\Vert H_{2}(\cdot,z) \Vert_{L^{1}} \leq C \Vert (v, u)(z) \Vert_{2}^{2} \leq C \Vert U(z) \Vert_{2}^{2} \leq C\Vert U(z) \Vert_{s-1}^{2},
\end{equation} 
for all $z \in [0,t]$ because $s\geq 3$, where $C > 0$ is a uniform constant. Consequently, estimates \eqref{nonestHs-1}, \eqref{H2-s} and \eqref{H2-L1} yield
\begin{equation}\label{nonestHs-1.2}
\begin{split}
\Vert U(t) \Vert_{s-1} &\leq C \left(1+t \right)^{-1/4} \left( \Vert U_{0}\Vert_{s-1}+ \Vert U_{0} \Vert_{L^{1}} \right) \:+ \\ & \:\:\:\: + C\sup_{0 \leq z \leq t} \Vert v(z) \Vert_{s} \int_{0}^{t}e^{-c_{1}(t-z)}\Vert v(z) \Vert_{s}\:dz \:+ \\ &\:\:\:\: +C \left( \int_{0}^{t}\Vert u_{x}(z)\Vert_{s}^{2}\: dz \right)^{1/2} \left( \int_{0}^{t}e^{-2c_{1}(t-z)}\Vert v(z) \Vert_{s}^{2}\: dz \right)^{1/2}\: + \\ &\:\:\:\:+ C\left( \int_{0}^{t}\Vert v_{x}(z) \Vert_{s+1}^{2}\: dz \right)^{1/2} \left( \int_{0}^{t}e^{-2c_{1}(t-z)}\Vert v(z) \Vert_{s}^{2} \: dz \right)^{1/2}\: + \\ &\:\:\:\: + C\int_{0}^{t}\left( 1+ t-z\right)^{-3/4} \Vert U(z) \Vert_{s-1}^{2} \: dz.
\end{split}
\end{equation}
In order to simplify the notation, let us define 
\[
E_{s}(t) := 
\sup_{0 \leq z \leq t} \left( 1+ z \right)^{1/4}\Vert U(z) \Vert_{s-1}.
\]
Hence, we obtain 
\begin{equation}
\label{nonestHs-1.sup}
E_{s}(t) \leq C \big(\Vert U_{0} \Vert_{s-1} + \Vert U_{0} \Vert_{L^{1}} \big) + C I_{1}(t)\vertiii{U}_{s,t} E_{s}(t) + C I_{2}(t) E_{s}(t)^2,
\end{equation}
where the norm $\vertiii\cdot{}_{s,t}$ is defined in \eqref{tripnorm} and
\begin{equation}\label{mu1}
\begin{split}
I_{1}(t) &:= \sup_{0\leq z \leq t}\left( 1+ z\right)^{1/4} \int_{0}^{z}e^{-c_{1}(z-z_{1})}(1+z_{1})^{-1/4}\: dz_{1} \: + \\ & \:\:\:\: + \sup_{0\leq z \leq t}\left( 1 + z \right)^{1/4}\left[ \int_{0}^{z}e^{-2c_{1}(z-z_{1})}(1+z_{1})^{-1/2} \: dz_{1} \right]^{1/2},
\end{split}
\end{equation}
\begin{equation}\label{mu2}
I_{2}(t) := \sup_{0 \leq z \leq t}\left( 1+ z \right)^{1/4}\int_{0}^{z}\left( 1+z-z_{1} \right)^{-3/4}(1+z_{1})^{-1/2} \: dz_{1}.
\end{equation}
Since both $I_{1}(t)$ and $I_{2}(t)$ are uniformly bounded in $t$ (see Lemma \ref{lemintbded}), we readily obtain the estimate
\begin{equation}
\label{finnonest}
E_{s}(t) \leq C\big(\Vert U_{0} \Vert_{s-1} + \Vert U_{0} \Vert_{L^{1}} \big) + C \vertiii{U}_{s,t} E_{s}(t) +  C E_{s}(t)^2.
\end{equation}
Hence, combining Theorem \ref{thmlocalepert} and estimate \eqref{finnonest}, we have proved the following

\begin{theorem}[nonlinear decay estimate]
\label{thnonest}
Under hypotheses \hyperref[H1]{(H$_1$)} - \hyperref[H3]{(H$_3$)}, let $s\geq 3$ and suppose that $ \left(v_{0}, u_{0} \right)^\top \in \left( H^{s+1}(\R) \times H^{s}(\R)\right) \cap \left( L^{1}(\R) \times L^{1}(\R ) \right)$ satisfies the smallness assumption (condition \eqref{smallpert}) of Theorem \ref{thmlocalepert} as a perturbation of a constant equilibrium state $\bU =  (\bv, \bu)^\top$. Let $(v+\bv,u+ \bu)(x, t)$ be the solution to \eqref{nlop2} and \eqref{InCondNlop} from Theorem \ref{thmlocalepert}, satisfying $(v+\bv, u+\bu)\in X_{s}(0,T;m,M)$. Then there exist constants $ 0 < a_{1} \leq a_{0}$, $\delta_{1}= \delta_{1}(a_{1}) > 0$ and $C_{1}=C_{1}(a_{1},\delta_{1})>1$ such that if $\vertiii{U}_{s,T} \leq a_{1}$ and $\Vert U_{0} \Vert_{s-1} + \Vert U_{0} \Vert_{L^{1}} < \delta_1$, then the following inequality 
\begin{equation}
\label{nonlest}
\Vert U(t) \Vert_{s-1} \leq C_{1} \left( 1 + t \right)^{-1/4} \big(\Vert U_{0} \Vert_{s-1} + \Vert U_{0} \Vert_{L^{1}} \big),
\end{equation}
holds for every $t\in [0, T]$.
\end{theorem}

\subsection{Global existence and decay rate of perturbations}

Finally, we conclude our analysis with the following global existence and time decay result for the solutions of the Cauchy problem under consideration. The proof follows a standard continuation method.
%
%
%
\begin{theorem}[global existence and asymptotic decay]
\label{gloexth} 
Under hypotheses \hyperref[H1]{(H$_1$)} - \hyperref[H3]{(H$_3$)}, suppose that $(v_{0} , u_{0})^\top \in \left( H^{s+1}(\R) \times H^{s}(\R)\right) \cap \left( L^{1}(\R) \times L^{1}(\R) \right)$ for $s\geq 3$. Then there exists a positive constant $\delta_{2} (\leq \delta_{1},a_{0})$ such that if $\Vert U_{0} \Vert_{s} + \Vert U_{0} \Vert_{L^{1}}  \leq \delta_{2}$, then the Cauchy problem \eqref{nlop2}, \eqref{InCondNlop} has a unique solution $(v+ \bv, u + \bu)(x,t)$ satisfying 
\begin{equation}\label{globsol}
\begin{split}
& \:\: v \in C\left((0,\infty);H^{s+1}(\mathbb{R})) \cap C^{1}((0,\infty); H^{s-1}(\mathbb{R})\right), \\ & \:\: u \in C\left((0,\infty); H^{s}(\mathbb{R})) \cap C^{1}((0,\infty);H^{s-2}(\mathbb{R})\right) \\ & \:\:(v_{x},u_{x})\in L^{2}\left((0,\infty);H^{s+1}(\mathbb{R}) \times H^{s}(\mathbb{R})\right).
\end{split}
\end{equation}
Furthermore, the solution satisfies the following estimates
\begin{equation}\label{trinormest}
\vertiii{U}_{s,t} \leq C_{2} \Vert U_{0} \Vert_{s},
\end{equation}
and
\begin{equation}\label{globdec}
\Vert U(t) \Vert_{s-1} \leq C_{1} \left(1+ t \right)^{-1/4} (\Vert U_{0} \Vert_{s-1} + \Vert U_{0} \Vert_{L^{1}} ),
\end{equation}
for every $t \in [0, \infty)$.
\end{theorem}
\begin{proof}
Let us take $a_{2} \leq a_{1}$, with $a_{1}$ and $a_{2}$ as in Theorem \ref{thnonest} and Corollary \ref{thaprest}, respectively. Then we define
\[
 \delta_{2} := \mbox{min} \left\lbrace a_{2},\frac{a_{2}}{C_{0}} , \frac{a_{2}}{C_{2}\left( 1+C_{0}^{2}\right)^{1/2}}, \delta_{1} \right\rbrace.
 \]
Let us assume $\Vert U_{0} \Vert_{s} + \Vert U_{0} \Vert_{L^{1}} \leq \delta_{2}$ is satisfied. Then we have $\Vert U_{0} \Vert_{s} \leq \delta_{2}\leq a_{2} \leq a_{0}$. Therefore, by the local existence of solutions, there exists a constant $T_{1}= T_{1}(a_{0})$ such that the solution exists on $X_{s}(0,T_{1};\frac{m_{0}}{2}, 2M_{0})$, with $\vertiii{U}_{s,T_{1}} \leq C_{0}\delta_{2} \leq a_{2}\leq a_{1}$ and $\Vert U_{0} \Vert_{s-1} + \Vert U_{0} \Vert_{L^{1}}\leq \delta_{2}\leq \delta_{1}$ by the definition of $\delta_{2}$. From $\vertiii{U}_{s,T_{1}}\leq a_{2}$, we have
\[
 \sup_{0 \leq t \leq T_{1}}  \Vert  U (t)\Vert_{s} \leq \vertiii{U}_{s,T_{1}} \leq a_{2}. 
\] 
 
Thus, the conditions of Theorem \ref{thnonest} are satisfied and estimate \eqref{aprest} holds. Therefore we obtain
\[
\Vert U(t) \Vert_{s-1} \leq C_{1}\left( 1 + t \right)^{-1/4} (\Vert U_{0} \Vert_{s-1} + \Vert U_{0} \Vert_{L^{1}}),\quad \mbox{for}\:\: t\in [0,T_{1}],
\]
and
\[
\vertiii{U}_{s,T_{1}} \leq C_{2}\Vert U_{0} \Vert_{s}. 
\]
Since $ \Vert U(T_{1}) \Vert_{s}  \leq \vertiii{U}_{s,T_{1}} \leq a_{2} \leq a_{0}$, we can consider the Cauchy problem \eqref{nlop2}, \eqref{InCondNlop} taking the initial time equal to $T_{1}$ and then find a solution in $[T_{1}, 2T_{1}]$ with the estimate $\vertiii{U}_{s,[T_{1},2T_{1}]} \leq C_{0} \Vert U(T_{1}) \Vert_{s}$. Hence, we have 
\begin{align*}
\vertiii{U}_{s,2T_{1}} &= \left( \vertiii{U}_{s,T_{1}}^{2}+\vertiii{U}_{s,[T_{1},2T_{1}]}^{2} \right)^{1/2} \\ & \leq \left(1+ C_{0}^{2} \right)^{1/2}\vertiii{U}_{s,T_{1}} \\&\leq C_{2}\left(1+ C_{0}^{2} \right)^{1/2} \Vert U_{0} \Vert_{s} \\ 
& \leq C_{2}\left(1 + C_{0}^{2} \right)^{1/2} (\Vert U_{0} \Vert_{s} + \Vert U_{0} \Vert_{L^{1}}) \\
&\leq C_{2}(1+C_{0}^{2})^{1/2} \delta_{2} \leq a_{2} \leq a_{1}. 
\end{align*}
The last estimate implies that
\[
\sup_{0 \leq t \leq 2T_{1}}  \Vert U(t) \Vert_{s} \leq \vertiii{U}_{s,2T_{1}} \leq a_{2}. 
\]
The above estimate (together with the already verified condition $\Vert U_{0} \Vert_{s-1} +\|U_0\|_{L^1}  \leq \delta_{1}$) enables us to obtain the estimates \eqref{nonlest} and \eqref{aprest} for $t\in [0, 2T_{1}]$, yielding
\[
\Vert U(t) \Vert_{s-1} \leq C_{1}\left( 1 + t \right)^{-1/4} (\Vert U_{0} \Vert_{s-1} + \Vert U_{0} \Vert_{L^{1}}), \quad \mbox{for}\:\: t \in [0, 2T_{1}],
\]
and
$$ \vertiii{U}_{s,2T_{1}} \leq C_{2}\Vert U_{0} \Vert_{s} ,$$
hold. We can repeat the argument and obtain the estimates \eqref{trinormest} and \eqref{globdec} for the time interval $[0, 3T_{1}]$ and so on. The proof is complete.
\end{proof}

\section*{Acknowledgements}

The authors thank Jeffrey Humpherys for useful discussions. The authors are also grateful to Prof. H. Hattori for providing a copy of the article \cite{HaLi96b}. The work of J. M. Valdovinos was partially supported by CONACyT (Mexico), through a scholarship for graduate studies, grant no. 712874. The work of R. G. Plaza was partially supported by DGAPA-UNAM, program PAPIIT, grant IN-104922.

\appendix
\section{Integral estimates}

For completeness, in this Appendix  we show that the time-dependent integrals defined in \eqref{I0}, \eqref{mu1} and in \eqref{mu2}, are uniformly bounded.

\begin{lemma}
\label{lemintbded}
There exists a uniform constant $C > 0$ such that
\[
\begin{aligned}
I_0(t) &= (1+t)^{\ell + 1/2} \int_{-1}^1 \xi^{2 \ell} e^{-kt \xi^2} \, d\xi \leq C,\\
I_1(t) &= \sup_{0\leq z \leq t}\left( 1+ z\right)^{1/4} \int_{0}^{z}e^{-c_{1}(z-\tau)}(1+ \tau)^{-1/4}\: d\tau \: + \\ & \:\:\:\: + \sup_{0\leq z \leq t}\left( 1 + z \right)^{1/4}\left[ \int_{0}^{z}e^{-2c_{1}(z - \tau)}(1 + \tau)^{-1/2} \: d\tau \right]^{1/2} \leq C,
\\
I_2(t) &= \sup_{0 \leq z \leq t} \left( 1+ z \right)^{1/4}\int_{0}^{z}\left( 1+z-\tau \right)^{-3/4}(1+\tau)^{-1/2} \, d \tau \leq C
\end{aligned}
\]
for all $t \geq 0$ and fixed $0 \leq \ell \leq s-1$, $s \geq 2$.
\end{lemma}
\begin{proof}
The claims for $I_0(t)$ and $I_1(t)$ follow from standard calculus tools. For instance, $I_0(t)$ is clearly continuous and non negative for all $t \geq 0$. Notice that $I_0(0) = 2/(2\ell+ 1) > 0$ for $\ell \geq 0$. Moreover, $I_0(t)$ is uniformly bounded in any compact interval $t \in [0,R]$ for any $R > 0$, with $I_0(t) \leq C_R$, for some $C_R > 0$. Making the change of variables $y = \xi^2 t$ for $t > 0$ we get
\[
I_0(t) = \Big( 1 + \frac{1}{t}\Big)^{\ell + 1/2} \int_0^{t} y^{\ell -1/2} e^{-ky} \, dy.
\]
Hence, it suffices to show that
\[
\lim_{t \to \infty} \int_0^{t} y^{\ell -1/2} e^{-ky} \, dy,
\]
exists. Observe that for $R \gg 1$, fixed but sufficiently large, if $y \geq R$ then $y^{\ell -1/2} \leq C e^{ky/2}$ for some $C > 0$. Thus,
\[
\int_R^t y^{\ell -1/2} e^{-ky} \, dy \leq C \int_R^t e^{-ky/2} \, dy \leq \frac{2C}{k} e^{-kR/2}.
\]
Hence, for $t \geq R$ with $R \gg 1$ we obtain 
\[
\begin{aligned}
I_0(t) &= \Big( 1 + \frac{1}{t}\Big)^{\ell + 1/2} \Big[ \int_0^R y^{\ell -1/2} e^{-ky} \, dy + \int_R^t y^{\ell -1/2} e^{-ky} \, dy \Big] \\
&\leq \Big( 1 + \frac{1}{R}\Big)^{\ell + 1/2} \Big[ \int_0^R y^{\ell -1/2} e^{-ky} \, dy + \frac{2C}{k} e^{-kR/2} \Big] = C_R.
\end{aligned}
\]
This shows the claim. Thanks to the exponential term inside the integrals, the proof of the bound for $I_1(t)$ is analogous and we omit it.

To prove the bound for $I_2(t)$ it suffices to show that 
\[
A_2(z) = \left( 1+ z \right)^{1/4}\int_{0}^{z}\left( 1+z-\tau \right)^{-3/4}(1+\tau)^{-1/2} \, d \tau,
\]
is uniformly bounded as $z \to \infty$. Make the change of variables, 
\[
y = (1+z-\tau)^{1/4}/(2+z)^{1/4},
\] 
inside the integral in order to obtain
\[
\begin{aligned}
A_2(z) &= 4 \frac{(1+z)^{1/4}}{(2+z)^{1/4}} \int_{(2+z)^{-1/4}}^{(1+z)^{1/4}/(2+z)^{1/4}} \frac{dy}{\sqrt{1-y^2} \sqrt{1+y^2}}\\
& \to 4 \int_0^1 \frac{dy}{\sqrt{1-y^2} \sqrt{1+y^2}} = 4 F( \tfrac{\pi}{2} \, | -\!1),
\end{aligned}
\]
as $z \to \infty$, where $F$ is the incomplete elliptic integral of the first kind (cf. \cite{ByFr71,Lawd89}),
\[
F(\varphi \, | \, k) = \int_0^{\sin \varphi} \frac{dy}{\sqrt{1-y^2} \sqrt{1-ky^2}}.
\]
By elementary properties of elliptic integrals we observe that $A_2(z)$ has a finite limit when $z \to \infty$ (in fact, $A_2 \to 4 F(\tfrac{\pi}{2} \, | -\!1) \approx 5.2441$) and, since $A_2(z)$ is continuous in any compact interval $z \in [0,R]$, we obtain the result.
\end{proof}

\def\cprime{$'$}





\end{document}